\documentclass[11pt,a4paper,oneside]{article}

\usepackage{authblk}

\title{On the number of residues of certain\\ second-order linear recurrences%
    \footnote{%
        The authors are members of CrypTO, the group of Cryptography and Number Theory of the \mbox{Politecnico di Torino.}
        C.~Sanna is a member of GNSAGA of INdAM.
}}

\author{Federico Accossato}
\author{Carlo Sanna}

\affil{%
    Department of Mathematical Sciences, Politecnico di Torino\\
    Corso Duca degli Abruzzi 24, 10129 Torino, Italy
}

\affil[]{\texttt{\{federico.accossato,carlo.sanna\}@polito.it}}

\date{} 

\usepackage{amsmath}
\usepackage{amssymb}
\usepackage{amsthm}
\usepackage{booktabs}
\usepackage{placeins}
\usepackage{geometry}
\usepackage{color}
\usepackage{hyperref}
\hypersetup{colorlinks=true}
\usepackage{enumitem}
\setlist[enumerate]{label=(\roman*),labelindent=1em,itemsep=0.5em,topsep=0.5em}

\newtheorem{theorem}{Theorem}[section]
\newtheorem{corollary}{Corollary}[section]
\newtheorem{lemma}[theorem]{Lemma}
\theoremstyle{remark}
\newtheorem{remark}{Remark}[section]

\usepackage{bm}

\DeclareMathOperator{\lcm}{lcm}
\DeclareMathOperator{\ord}{ord}
\DeclareMathOperator{\floorpart}{floor}
\DeclareMathOperator{\fracpart}{frac}

\begin{document}

\maketitle

\begin{abstract}
    For every monic polynomial $f \in \mathbb{Z}[X]$ with $\deg(f) \geq 1$, let
    $\mathcal{L}(f)$ be the set of all linear recurrences with values in $\mathbb{Z}$ and characteristic polynomial $f$, and let
    \begin{equation*}
	    \mathcal{R}(f) := \big\{\rho(\bm{x}; m) : \bm{x} \in \mathcal{L}(f), \, m \in \mathbb{Z}^+ \big\} ,
	\end{equation*}
    where $\rho(\bm{x}; m)$ is the number of distinct residues of $\bm{x}$ modulo $m$.

    Dubickas and Novikas proved that $\mathcal{R}(X^2 - X - 1) = \mathbb{Z}^+$.
    We generalize this result by showing that $\mathcal{R}(X^2 - a_1 X - 1) = \mathbb{Z}^+$ for every nonzero integer $a_1$.
    As a corollary, we deduce that for all integers $a_1 \geq 1$ and $k \geq 4$ there exists $\xi \in \mathbb{R}$ such that the sequence of fractional parts $\big(\!\fracpart(\xi \alpha^n)\big)_{n \geq 0}$, where $\alpha := \big(a_1 + \sqrt{a_1^2 + 4}\,\big) / 2$, has exactly $k$ limit points.
    Our proofs are constructive and employ some results on the existence of special primitive divisors of certain Lehmer sequences.
\end{abstract}

\section{Introduction}

Let $a_1, \dots, a_k$ be integers.
An integer sequence $\bm{x} = (x_n)_{n\geq 0}$ is a \emph{linear recurrence} with \emph{characteristic polynomial}
\begin{equation*}
	f = X^k - a_1 X^{k - 1} - a_2 X^{k - 2} - \cdots - a_k
\end{equation*}
if for all integers $n \geq k$ we have that
\begin{equation}\label{equ:recurrence}
	x_n = a_1 x_{n - 1} + a_2 x_{n - 2} + \cdots + a_k x_{n - k} .
\end{equation}
The terms $x_0, \dots, x_{k-1}$, which together with $f$ completely determine $\bm{x}$ via~\eqref{equ:recurrence}, are the \emph{initial values} of $\bm{x}$.
We let $\mathcal{L}(f)$ denote the set of all (integral) linear recurrences with characteristic polynomial $f$.
It is easily seen that each $\bm{x} \in \mathcal{L}(f)$ is ultimately periodic modulo $m$, for every integer $m \geq 1$, and in fact (purely) periodic if $\gcd(m, a_k) = 1$.
Indeed, properties of linear recurrences modulo $m$ have been studied extensively, including: which residues modulo $m$ appear in $\bm{x}$ and how frequently~\cite{MR369240,MR3298566,MR2024599,MR1119645,MR1117027,MR1393478}, and for which values of $m$ the linear recurrence $\bm{x}$ contains a complete system of residues modulo $m$~\cite{MR3068231,MR294238,MR4357629,MR951911,MR3696266}.

We let $\tau(\bm{x}; m)$ denote the (minimal) period of $\bm{x}$ modulo $m$, that is, the minimal integer $t \geq 1$ such that $x_{n + t} \equiv x_n \pmod m$ for all sufficiently large integers $n \geq 0$.
Moreover, we let $\rho(\bm{x}; m) := \#\big\{x_n \bmod m : n \geq 0\big\}$ be the number of distinct residues of $\bm{x}$ modulo $m$, and we put
\begin{equation*}
	\mathcal{R}(f) := \big\{\rho(\bm{x}; m) : \bm{x} \in \mathcal{L}(f), \, m \in \mathbb{Z}^+ \big\} .
\end{equation*}
Dubickas and Novikas~\cite{MR4125906}, motivated by some problems on fractional parts of powers of Pisot numbers~\cite{MR3397268}, proved that $\mathcal{R}(X^2 - X - 1) = \mathbb{Z}^+$ and stated that it ``may be very diﬃcult in general'' to determine $\mathcal{R}(f)$.
Sanna~\cite{MR4350306} considered the special case in which $f$ is a quadratic polynomial with roots $\alpha, \beta$ such that $\alpha\beta = \pm 1$ and $\alpha/\beta$ is not a root of unity, and proved two results.
First, that $\mathcal{R}(f)$ contains all integers $n \geq 7$ with $n \neq 10$ and $4 \nmid n$.
Second, that $4D_0 \mathbb{Z}^+ \subseteq \mathcal{R}(f)$ if $\alpha\beta = 1$, and $8D_0 \mathbb{Z}^+ \subseteq \mathcal{R}(f)$ if $\alpha\beta = -1$; where $D_0$ is the squarefree part of the discriminant of $f$, and it is assumed that $D_0 \equiv 1 \pmod 4$ and $D_0 \geq 5$.

Our result is the following.

\begin{theorem}\label{thm:main}
	Let $a_1$ be a nonzero integer.
	Then $\mathcal{R}(X^2 - a_1 X - 1) = \mathbb{Z}^+$.
    In other words, for every $n \in \mathbb{Z}^+$ there exist $\bm{x} \in \mathcal{L}(f)$ and $m \in \mathbb{Z}^+$ such that $\rho(\bm{x}; m) = n$.
    Moreover, one can choose $\bm{x}$ and $m$ so that all the residues of $\bm{x}$ modulo $m$ are nonzero if and only if $|a_1| = 1$ and $n \geq 4$, or $|a_1| \geq 2$.
\end{theorem}

The assumption that $a_1$ is nonzero is not a restriction, since it is easy to prove that $\mathcal{R}(X^2 - 1) = \{1, 2\}$.
We remark that the proof of Theorem~\ref{thm:main} is constructive.
It provides an algorithm that, given as input a nonzero integer $a_1$ and an integer $n \geq 1$, returns as output a modulo $m$ and the initial values $x_0, x_1$ of a linear recurrence $\bm{x} \in \mathcal{L}(X^2 - a_1 X - 1)$ such that $\rho(\bm{x}; m) = n$ and, if $|a_1| \geq 2$ or $n \geq 4$, all the residues of $\bm{x}$ modulo $m$ are nonzero.

For every $t \in \mathbb{R}$, let $\floorpart(t)$ be the \emph{floor function} of $t$, that is, the greatest integer not exceeding $t$, and let $\fracpart(t) := t - \floorpart(t)$ be the \emph{fractional part} of $t$.
As a corollary of their aforementioned result, Dubickas and Novikas~\cite{MR4125906} proved that for every integer $k \geq 2$ there exists $\xi \in \mathbb{R}$ such that the sequence $\big(\!\fracpart\!\big(\xi ((1 + \sqrt{5})/2)^n\big)\big)_{n \geq 0}$ has exactly $k$ limit points.

From Theorem~\ref{thm:main}, we deduce the following corollary.

\begin{corollary}\label{cor:fractional-parts}
    Let $a_1 \geq 1$ be an integer, and let $\alpha := \big(a_1 + \sqrt{a_1^2 + 4}\,\big) / 2$.
    Then, for every integer $k \geq 2$, and for every integer $k \geq 1$ if $a_2 \geq 2$, there exists $\xi \in \mathbb{Q}(\alpha)$ such that the sequence $\big(\!\fracpart(\xi \alpha^n)\big)_{n \geq 0}$ has exactly $k$ limit points.
\end{corollary}

We remark that sequence of fractional parts $\big(\!\fracpart(\xi \alpha^n)\big)_{n \geq 0}$, where $\alpha > 1$ is a real algebraic number and $\xi \in \mathbb{R}$, has been studied by several authors~\cite{MR2184199,MR2201605,MR2240352,MR0017744,MR0002326,MR2256803,MR3397268}.

The paper is structured as follows.
Section~\ref{sec:preliminaries} contains several preliminary lemmas.
More precisely, Section~\ref{sec:lehmer} is devoted to prove the existence of some special primitive divisors of specific Lehmer sequences, Section~\ref{sec:orders} contains some elementary but useful results on certain multiplicative orders, and Section~\ref{sec:second-order-recurrences} provides the main lemmas for the construction of the desired sequences in $\mathcal{L}(f)$.
Then, Sections~\ref{sec:proof-main} and~\ref{sec:proof-fractional-parts} are devoted to the proofs of Theorem~\ref{thm:main} and Corollary~\ref{cor:fractional-parts}, respectively.

\section{Preliminaries}\label{sec:preliminaries}

Hereafter, let $a_1$ be a nonzero integer, let $f := X^2 - a_1 X - 1$, let $D := a_1^2 + 4$ be the discriminant of $f$, let $K := \mathbb{Q}\big(\!\sqrt{D}\big)$ be the splitting field of $f$, and let $\alpha, \beta \in K$ be the roots of $f$.
Note that $D$ is not a square in $\mathbb{Q}$, so that $K$ is a quadratic number field.
Moreover, note that $\alpha$ and $\beta$ are distinct real numbers, $\alpha\beta=-1$, $\alpha + \beta = a_1$, $(\alpha-\beta)^2 = D$, and $\alpha / \beta$ is not a root of unity.

We begin with the following lemma.

\begin{lemma}\label{lem:reversal}
    Let $\bm{x} \in \mathcal{L}(f)$, let $m \geq 1$ be an integer, and let $t := \tau(\bm{x}; m)$.
    Then there exists $\bm{y} \in \mathcal{L}(X^2 + a_1 X - 1)$ such that $y_n \equiv x_{(-(n + 1) \bmod t)}\! \pmod m$ for every integer $n \geq 0$, where $(-(n+1) \bmod t)$ is the unique integer $r \in {[0,t)}$ such that $r \equiv -(n+1) \pmod t$.
    In~particular, we have that $\tau(\bm{x}; m) = \tau(\bm{y}; m)$, $\rho(\bm{x}; m) = \rho(\bm{y}; m)$, and that all residues of $\bm{x}$ modulo $m$ are nonzero if and only if all residues of $\bm{y}$ modulo $m$ are nonzero.
\end{lemma}
\begin{proof}
	Let us extend $\bm{x}$ to negative indices by defining $x_n := -a_1 x_{n + 1} + x_{n + 2}$ for every integer $n \leq -1$.
	Then, we have that $x_n = -a_1 x_{n + 1} + x_{n + 2}$ and $x_{n + t} \equiv x_n \pmod m$ for every integer $n$.
	Let $\bm{y} \in \mathcal{L}(X^2 + a_1 X - 1)$ with $y_0 = x_{t - 1}$ and $y_1 = x_{t - 2}$.
	Since $y_{n + 2} = -a_1 y_{n + 1} + y_n$ for every integer $n \geq 0$, it follows easily by induction that $y_n = x_{t - (n + 1)}$ for every integer $n \geq 0$.
	Hence, by the periodicity of $\bm{x}$, we get that $y_n \equiv x_{(-(n + 1) \bmod t)}\! \pmod m$ for every integer $n \geq 0$, as desired.
\end{proof}

In light of Lemma~\ref{lem:reversal}, hereafter we assume that $a_1 \geq 1$.

\begin{remark}
    Since Dubickas and Novikas already proved the case $a_1 = 1$ of Theorem~\ref{thm:main} and Corollary~\ref{cor:fractional-parts}, we could assume that $a_1 \geq 2$.
    However, we choose to include the case $a_1 = 1$ in our proofs in order to state some of the intermediary results with greater generality.
\end{remark}

\subsection{Lehmer sequences}\label{sec:lehmer}

Let $\gamma, \delta$ be complex numbers such that $\gamma\delta$ and $(\gamma + \delta)^2$ are nonzero coprime integers and $\gamma / \delta$ is not a root of unity.
The \emph{Lehmer sequence} $(\ell_n)$ associated to $\gamma, \delta$ is defined by
\begin{equation}\label{equ:lehmer-definition}
	\ell_n :=
	\begin{cases}
		(\gamma^n - \delta^n) / (\gamma - \delta) & \text{ if $n$ is odd}; \\
		(\gamma^n - \delta^n) / (\gamma^2 - \delta^2) & \text{ if $n$ is even};
	\end{cases}
\end{equation}
for all integers $n \geq 0$.
The conditions on $\gamma, \delta$ ensure that each $\ell_n$ is an integer.
A prime number $p$ is a \emph{primitive divisor} of $\ell_n$ if $p \mid \ell_n$ but $p \nmid (\gamma^2 - \delta^2)^2 \ell_1 \cdots \ell_{n-1}$ (cf.~\cite{MR1863855}).
Note that $\ell_1 = \ell_2 = 1$, so that $\ell_n$ can have primitive divisors only if $n \geq 3$.

The sequence of \emph{cyclotomic numbers} $(\phi_n)$ associated to $\gamma, \delta$ is defined by
\begin{equation}\label{equ:cyclotomic-definition}
	\phi_n := \prod_{\substack{1 \,\leq\, k \,\leq\, n \\[1pt] \gcd(n,k) \,=\, 1}} \left(\gamma - \mathrm{e}^{\frac{2\pi\mathbf{i}k}{n}} \delta\right) ,
\end{equation}
for all integers $n \geq 1$.
It can be proved that $\phi_n \in \mathbb{Z}$ for every integer $n \geq 3$ \cite[p.~84]{MR1863855}.

The next lemma relates the cyclotomic numbers with the primitive divisors of the Lehmer sequence.
For every prime number $p$, let $\nu_p$ denote the $p$-adic valuation.
For every integer $n \geq 4$, let $P^\prime(n)$ denote the greatest prime factor of $n /\! \gcd(n, 3)$.

\begin{lemma}\label{lem:product-of-primitive-divisors}
	Let $n \geq 5$ be an integer with $n \neq 6$.
	Then
	\begin{equation*}
		|\phi_n| \, \Big/ \!\! \prod_{p \emph{ prim.\ div.\ } \ell_n} p^{\nu_p(\ell_n)}
		\in
		\begin{cases}
			\{1, P^\prime(n)\} & \emph{ if } n \neq 12 ; \\
			\{1, 2, 3, 6\} & \emph{ if } n = 12 ; \\
		\end{cases}
	\end{equation*}
	where the product runs over all the primitive divisors $p$ of $\ell_n$.
\end{lemma}
\begin{proof}
	From~\eqref{equ:cyclotomic-definition} it follows easily that
	\begin{equation}\label{equ:binomial}
		\gamma^m - \delta^m = \prod_{d \,\mid\, m} \phi_d ,
	\end{equation}
	for all integers $m \geq 1$.
	In turn, applying to~\eqref{equ:binomial} the M\"obius inversion formula and taking into account~\eqref{equ:lehmer-definition}, we get that
	\begin{equation}\label{equ:cyclotomic-as-lehmer-product}
		\phi_m = \prod_{d \,\mid\, m} \left(\gamma^d - \delta^d\right)^{\mu(m / d)} = \prod_{d \,\mid\, m} \ell_d^{\,\mu(m / d)} ,
	\end{equation}
	for all integers $m \geq 3$, where $\mu$ is the M\"obius function and where we also employed the well-known fact that $\sum_{d \,\mid\, m} \mu(d) = 0$ for every integer $m \geq 2$.

	Let $p$ be a prime number.
	If $p$ is a primitive divisor of $\ell_n$, then $p$ does not divide $\ell_k$ for every positive integer $k < n$.
	Hence, from~\eqref{equ:cyclotomic-as-lehmer-product} we have that $\nu_p(\phi_n) = \nu_p(\ell_n)$.

	If $p$ is not a primitive divisor of $\ell_n$ and $p \mid \phi_n$, then it is known that $\nu_p(\phi_n) = 1$ and that either $p = P^\prime(n)$, if $n \neq 12$, or $p \in \{2, 3\}$, if $n = 12$; see the ``only if'' part of the proof of \cite[Theorem~2.4]{MR1863855}, or see the paragraphs before, and the comment after, \cite[Lemma~6]{MR0491445}.
\end{proof}

We need a lower bound for the absolute values of cyclotomic numbers.

\begin{lemma}\label{lem:lower-bound-cyclotomic}
	Suppose that $\gamma, \delta$ are real numbers with $\gamma\delta > 0$.
	Then
	\begin{equation*}
		|\phi_n| > |4 (\gamma - \delta)^2 \gamma \delta|^{\varphi(n) / 4}
	\end{equation*}
	for every integer $n \geq 3$, where $\varphi$ is the Euler function.
\end{lemma}
\begin{proof}
	This result is due to Ward~\cite[p.~233]{MR0071446}.
\end{proof}

Hereafter, we put $\gamma := \alpha$ and $\delta := -\beta$.
Note that this choice does indeed satisfy the conditions of Lehmer sequence.
Moreover, we have that $\gamma\delta = 1$, $\gamma - \delta = a_1$, $(\gamma + \delta)^2 = D$, and $(\gamma^2 - \delta^2)^2 = a_1^2 D$.
Furthermore, the first values of $(\ell_n)$ are
\begin{align}\label{equ:small-lehmer}
    \ell_1 &= 1, \quad \ell_2 = 1, \quad \ell_3 = a_1^2 + 3, \quad \ell_4 = a_1^2 + 2, \nonumber\\
    \ell_5 &= (a_1^2 + 1)(a_1^2 + 4) + 1, \quad \ell_6 = (a_1^2 + 1)\ell_3 .
\end{align}
In particular, note that, since $\ell_3$ or $\ell_4$ is even, if $n \geq 5$ then every primitive divisor of $\ell_n$ (if it exists) is odd.

The problem of determining which terms of a Lehmer sequence have a primitive divisor has a very long history.
The first complete classification was given by Bilu, Hanrot, and Voutier~\cite{MR1863855} (see also~\cite[Remark~3.1]{MR4350306} for some missing values in \cite[Table~4]{MR1863855}).
We make use of the following two results.

\begin{lemma}\label{lem:existence-of-primitive-divisor}
	Let $n \geq 3$ be an integer.
	If $a_1 = 1$ and $n \notin \{3, 6, 10, 12\}$, or $a_1 = 3$ and $n \neq 3$, or $a_1 \notin \{1, 3\}$, then $\ell_n$ has at least one odd primitive divisor $p$.
\end{lemma}
\begin{proof}
	If $n \geq 13$ then it is known that $\ell_n$ has at least one odd primitive divisor~\cite[Lemma~8]{MR0491445}.
	(In fact, this is true for all Lehmer sequences with real $\gamma, \delta$.)
	Hereafter, assume that $n \leq 12$.

	If $a_1 = 1$ and $n \notin \{3, 6, 10, 12\}$, then one can check that $\ell_n$ has an odd primitive divisor.
	If $a_1 = 3$ and $n \neq 3$, then one can check that $\ell_n$ has an odd primitive divisor.

	Hereafter, assume that $a_1 \notin \{1, 3\}$.
	If $n \geq 5$ and $n \neq 6$, then by Lemma~\ref{lem:lower-bound-cyclotomic}, recalling that $\gamma - \delta = a_1$, $\gamma \delta = 1$, we get that $|\phi_n| > 2^{\varphi(n)} > n$.
	In turn, by Lemma~\ref{lem:product-of-primitive-divisors}, this implies that $\ell_n$ has an odd primitive divisor.

	It remains to prove that if $n \in \{3, 4, 6\}$ then $\ell_n$ has an odd primitive divisor.
	Since $a_1 \notin \{1, 3\}$, by reasoning modulo $8$ and modulo $9$, it follows easily that $a_1^2 + 2$ is not a power of $2$, and that neither $a_1^2 + 1$ nor $a_1^2 + 3$ is equal to $2^s 3^t$ for some integers $s, t \geq 0$.
	Hence, there exists a prime number $p \geq 5$ dividing $a_1^2 + 1$.
	We claim that $p$ is an odd primitive divisor of $\ell_6$.
	In fact, since $p \geq 5$, by \eqref{equ:small-lehmer} and recalling that $D = a_1^2 + 4$, we get that $p \mid \ell_6$ and $p \nmid a_1 D \ell_1 \cdots \ell_5$.
	The proof that $\ell_3$ and $\ell_4$ have an odd primitive divisor proceeds similarly.
\end{proof}

\begin{remark}
	Lemma~\ref{lem:existence-of-primitive-divisor} is optimal, that is, one can verify that if $a_1 = 1$ and $n \in \{3, 6, 10, 12\}$, or if $a_1 = 3$ and $n = 3$, then $\ell_n$ does not have an odd primitive divisor.
\end{remark}

For every integer $n \geq 1$, we say that a prime number $p$ is a \emph{high primitive divisor} of $\ell_n$ if $p$ is an odd primitive divisor of $\ell_n$ such that $p^{\nu_p(\ell_n)} > n$.

\begin{lemma}\label{lem:existence-of-primitive-divisor-high-power}
	Let $n \geq 3$ be an integer.
	If $a_1 = 1$ and $n \notin \{3, 4, 6, 8, 10, 12, 14, 18, 24\}$, or $a_1 = 2$ and $n \notin \{4, 6, 12\}$, or $a_1 = 3$ and $n \notin \{3, 6\}$, or $a_1 \geq 4$, then $\ell_n$ has at least one high primitive divisor.
\end{lemma}
\begin{proof}
    Suppose that $n \geq 5$ and $n \neq 6$.
	We claim that if $(2a_1)^{\varphi(n) / 2} > n^2$ then $\ell_n$ has a high primitive divisor.
	Indeed, recalling that $\gamma - \delta = a_1$ and $\gamma \delta = 1$, by Lemma~\ref{lem:lower-bound-cyclotomic} we have that $|\phi_n| > (2a_1)^{\varphi(n) / 2}$.
	Hence, by Lemma~\ref{lem:product-of-primitive-divisors}, we get that $(2a_1)^{\varphi(n) / 2} > n^2$ implies that
	\begin{equation}\label{equ:big-product}
		\prod_{q \text{ prim.\ div.\ } \ell_n} q^{\nu_q(\ell_n)} > n .
	\end{equation}
	In particular, from~\eqref{equ:big-product} it follows that $\ell_n$ has at least one primitive divisor.
	Let $p$ be the greatest primitive divisor of $\ell_n$.
	It is known that each primitive divisor $q$ of $\ell_n$ satisfies $q \equiv \pm 1 \pmod n$~\cite[p.~427]{MR0491445}.
	Hence, either $p \geq n + 1$, or $p = n - 1$ and $p$ is the unique primitive divisor of $\ell_n$.
	In both cases, taking into account~\eqref{equ:big-product}, we get that $p^{\nu_p(\ell_n)} > n$.
	Hence, we have that $p$ is a high primitive divisor of $\ell_n$ (note that $p$ is odd since $n \geq 5$).

	If $n \geq 2 \cdot 10^9$ then $\varphi(n) > n / \log n$ \cite[Lemma~4.1]{MR0071446}.
	Hence, we get that
	\begin{equation*}
		(2a_1)^{\varphi(n) / 2} > 2^{n / (2\log n)} > n^2 ,
	\end{equation*}
	which implies that $\ell_n$ has a high primitive divisor.

	If $180 \leq n  < 2 \cdot 10^9$, then $\varphi(n) > n / 6$ \cite[Lemma~4.2]{MR0071446}.
	Hence, we get that
	\begin{equation*}
		(2a_1)^{\varphi(n) / 2} > 2^{n / 12} > n^2 ,
	\end{equation*}
	which implies that $\ell_n$ has a high primitive divisor.

	If $5 \leq n \leq 179$ with $n \neq 6$, and $|a_1| \geq 7$, then
	\begin{equation*}
	(2a_1)^{\varphi(n) / 2} \geq 14^{\varphi(n) / 2} > n^2 ,
	\end{equation*}
	which implies that $\ell_n$ has a high primitive divisor.

	If $5 \leq n \leq 179$ with $n \neq 6$, and $a_1 \leq 6$, then with the aid of a computer one can check that $\ell_n$ has a high primitive divisor, unless $a_1 = 1$ and $n \in \{8, 10, 12, 14, 18, 24\}$, or $a_1 = 2$ and $n = 12$.
	Note that this verification is done more efficiently by factorizing $\ell_n$ only in the cases in which the inequality $(2a_1)^{\varphi(n) / 2} > n^2$ does not hold.

	If $a_1 = 2$ or $a_1 = 3$, then $\ell_3$ or $\ell_4$ has a high primitive divisor, respectively.
	Hereafter, assume that $n \in \{3, 4, 6\}$ and $a_1 \geq 4$.
	It remains to prove that $\ell_n$ has a high primitive divisor.
	This is done similarly to the end of the proof of Lemma~\ref{lem:existence-of-primitive-divisor}.
	Since $a_1 \geq 4$, by reasoning modulo $8$ and modulo $9$, it follows easily that $a_1^2 + 2$ is not equal to $2^s$ or $2^s 3$, for some integer $s \geq 0$, and that neither $a_1^2 + 1$ nor $a_1^2 + 3$ is equal to $2^s 3^t$ or $2^s 3^t 5$, for some integers $s, t \geq 0$.
	Hence, there exists a prime number $p \geq 5$ dividing $a_1^2 + 1$ and such that $p^{\nu_p(a_1^2 + 1)} > 6$.
	We claim that $p$ is a high primitive divisor of $\ell_6$.
	In fact, since $p \geq 5$, by \eqref{equ:small-lehmer} and recalling that $D = a_1^2 + 4$, we get that $p \mid \ell_6$, $p \nmid a_1 D \ell_1 \cdots \ell_5$, and $p^{\nu_p(\ell_6)} > 6$.
	The proof that $\ell_3$ and $\ell_4$ have a high primitive divisor proceeds similarly.
\end{proof}

\begin{remark}
	Lemma~\ref{lem:existence-of-primitive-divisor-high-power} is optimal, that is,
	one can verify that if $a_1 = 1$ and $n \in \{3, 4, 6, 8, 10, 12, 14, 18, 24\}$, or $a_1 = 2$ and $n \in \{4, 6, 12\}$, or $a_1 = 3$ and $n \in \{3, 6\}$, then $\ell_n$ has no high primitive divisor.
\end{remark}

For every odd prime number $p$, let $(D \mid p)$ denote the Legendre symbol.
Note that $(D \mid p) = 1$ if and only if $f$ splits modulo $p$.

\begin{lemma}\label{lem:legendre-D-p-equal-1}
	Let $n \geq 3$ be an odd integer and let $p$ be an odd prime factor of $\ell_n$.
	Then we have that $(D \mid p) = 1$.
\end{lemma}
\begin{proof}
	Since $n$ is odd, we have that $v_n := (\gamma^n + \delta^n) / (\gamma + \delta)$ is a symmetric expression of the algebraic integers $\alpha, \beta$ (recall that $\gamma = \alpha$ and $\delta = -\beta$).
	Hence, we get that $v_n$ is an integer.
	Furthermore, from $p \mid \ell_n$ and the identity
	\begin{equation*}
		D v_n^2 - a_1^2 \ell_n^2 = (\gamma + \delta)^2 v_n^2 - (\gamma - \delta)^2 \ell_n^2 = 4(\gamma\delta)^2 = 4,
	\end{equation*}
	we get that $D v_n^2 \equiv 2^2 \pmod p$.
	Since $p$ is odd, it follows that $p \nmid v_n$ and so $D \equiv (2 v_n^{-1})^2 \pmod p$.
	Thus $(D \mid p) = 1$, as desired.
\end{proof}

\subsection{Multiplicative orders}\label{sec:orders}

For every ideal $\mathfrak{i}$ of $\mathcal{O}_K$ and for every $\theta \in \mathcal{O}_K$, let $\ord_{\mathfrak{i}}(\theta)$ denote the multiplicative order of $\theta$ modulo $\mathfrak{i}$ (if it exists).

\begin{lemma}\label{lem:prim-div-order-alpha2}
	Let $n,v \geq 1$ be integers, let $p$ be a prime number not dividing $a_1 D$, and let $\mathfrak{p}$ be a prime ideal of $\mathcal{O}_K$ lying over $p$.
	Then $p^v \mid \ell_n$ if and only if $\ord_{\mathfrak{p}^v}(\alpha^2) \mid n$.
	In particular, we have that $p$ is a primitive divisor of $\ell_n$ if and only if $n = \ord_{\mathfrak{p}}(\alpha^2)$.
\end{lemma}
\begin{proof}
	Since $p \nmid a_1 D$, we have that $\alpha + \beta$ and $\alpha - \beta$ are invertible modulo $\mathfrak{p}^v$.
	Hence, using~\eqref{equ:lehmer-definition}, we get that $p^v \mid \ell_n$ is equivalent to $\alpha^n \equiv (-\beta)^n \pmod{\mathfrak{p}^v}$, which in turn is equivalent to $(\alpha^2)^n \equiv 1 \pmod{\mathfrak{p}^v}$, since $\alpha / (-\beta) = \alpha^2$.
	Thus the claim follows.
\end{proof}

\begin{lemma}\label{lem:sqrt-1}
	Let $p$ be an odd prime number, let $\mathfrak{p}$ be a prime ideal of $\mathcal{O}_K$ lying over $p$, and let $v \geq 1$ be an integer.
	Then the equation $X^2 \equiv 1 \pmod {\mathfrak{p}^v}$ has exactly two solutions modulo $\mathfrak{p}^v$, namely $1$ and $-1$.
\end{lemma}
\begin{proof}
	The equation $X^2 \equiv 1 \pmod {\mathfrak{p}^v}$ is equivalent to  $\nu_{\mathfrak{p}}(X - 1) + \nu_{\mathfrak{p}}(X + 1) \geq v$,
	where $\nu_{\mathfrak{p}}$ is the $\mathfrak{p}$-adic valuation over $\mathcal{O}_K$.
	Moreover, we have that $\nu_{\mathfrak{p}}(X - 1)$ and $\nu_{\mathfrak{p}}(X + 1)$ cannot both be positive, since $2 \notin \mathfrak{p}$.
	Hence, either $X \equiv 1 \pmod {\mathfrak{p}^v}$ or $X \equiv -1 \pmod {\mathfrak{p}^v}$.
\end{proof}

\begin{lemma}\label{lem:orders}
	Let $p$ be an odd prime number, let $\mathfrak{p}$ be a prime ideal of $\mathcal{O}_K$ lying over $p$, and let $v \geq 1$ be an integer.
	Put $a := \ord_{\mathfrak{p}^v}(\alpha)$, $b := \ord_{\mathfrak{p}^v}(\beta)$, and $c := \ord_{\mathfrak{p}^v}(\alpha^2)$.
	Then the following statements hold.
	\begin{enumerate}
		\item\label{ite:orders:1} If $c$ is odd then $a = c$ and $b = 2c$, or $a = 2c$ and $b = c$.
		\item\label{ite:orders:2} If $c$ is even then $a = b = 2c$.
		\item\label{ite:orders:3} $\lcm(a, b) = 2c$.
	\end{enumerate}
\end{lemma}
\begin{proof}
	Hereafter, all congruences are modulo $\mathfrak{p}^v$.
	Since $\alpha^{2c} \equiv 1$, we have that $a \mid 2c$ and, by Lemma~\ref{lem:sqrt-1}, that $\alpha^c \equiv \pm 1$.
	Moreover, from $\alpha^a \equiv 1$, we have that $\alpha^{2a} \equiv 1$, and so $c \mid a$.
	Hence, we get that $a = c$ if $\alpha^c \equiv 1$, and $a = 2c$ if $\alpha^c \equiv -1$.

	Using the fact that $\beta^2 = \alpha^{-2}$, we get that $c = \ord_{\mathfrak{p}^v}(\beta^2)$.
	Hence, by a reasoning similar to before, we have that $\beta^c \equiv \pm 1$, while $b = c$ if $\beta^c \equiv 1$, and $b = 2c$ if $\beta^c \equiv -1$.

	Suppose that $c$ is odd.
	Since $\alpha\beta = -1$, we have that $\alpha^c \beta^c \equiv (-1)^c \equiv -1$.
	Hence, either $\alpha^c \equiv 1$ and $\beta^c \equiv -1$, or $\alpha^c \equiv -1$ and $\beta^c \equiv 1$.
	By the previous considerations, it follows that either $a = c$ and $b = 2c$, or $a = 2c$ and $b = c$.
	This proves \ref{ite:orders:1}.

	Suppose that $c$ is even.
	Then, by the minimality of $c$, we get that $\alpha^c \equiv \beta^c \equiv -1$.
	Hence, by the previous considerations, we have that $a = b = 2c$.
	This proves \ref{ite:orders:2}.

	Finally, we obtain \ref{ite:orders:3} as a direct consequence of \ref{ite:orders:1} and \ref{ite:orders:2}.
\end{proof}

\subsection{Second order linear recurrences}\label{sec:second-order-recurrences}

In this section, we collect some results on linear recurrences in $\mathcal{L}(f)$.

\begin{lemma}\label{lem:binet}
	Let $\bm{x} \in \mathcal{L}(f)$.
	Then we have that
	\begin{equation*}
		x_n = \frac{(x_1 - \beta x_0) \alpha^n - (x_1 - \alpha x_0) \beta^n}{\alpha - \beta}
	\end{equation*}
	for every integer $n \geq 0$.
\end{lemma}
\begin{proof}
	This is a special case of the general expression of linear recurrences as generalized power sums~\cite[Sec.~1.1.6]{MR1990179} and can be easily proven by induction.
\end{proof}

\begin{lemma}\label{lem:period}
	Let $\bm{x} \in \mathcal{L}(f)$, let $p$ be a prime number not dividing $2(x_1^2 - a_1 x_1 x_0 - x_0^2) D$, let $\mathfrak{p}$ be a prime ideal of $\mathcal{O}_K$ lying over $p$, and let $v \geq 1$ be an integer.
	Then we have that $\tau(\bm{x}; p^v) = 2\ord_{\mathfrak{p}^v}(\alpha^2)$.
\end{lemma}
\begin{proof}
	The claim can be derived from more general results on the period of linear recurrences modulo integers \cite[Sec.~3.1]{MR1990179}, but we provide a short proof for completeness.

	Define the matrix
	\begin{equation*}
		M := \begin{pmatrix}
				(x_1 - \beta x_0) & -(x_1 - \alpha x_0) \\
				(x_1 - \beta x_0)\alpha & -(x_1 - \alpha x_0)\beta \\
			\end{pmatrix}
	\end{equation*}
	and note that
	\begin{equation*}
		\det(M) = (x_1^2 - a_1 x_1 x_0 - x_0^2) (\alpha - \beta) .
	\end{equation*}
	Since $p \nmid (x_1^2 - a_1 x_1 x_0 - x_0^2) D$, we get that $\alpha - \beta$ and $M$ are invertible modulo $\mathfrak{p}^v$.
	Hence, by Lemma~\ref{lem:binet}, for every integer $n \geq 1$, we have that
	\begin{align*}
		&\begin{cases}
			x_n &\!\!\!\!\equiv x_0 \\
			x_{n+1} &\!\!\!\!\equiv x_1 \\
		\end{cases} \pmod{p^v}
		\quad\Longleftrightarrow\quad
		M \begin{pmatrix}
			\alpha^n \\
			\beta^n
		\end{pmatrix}
		\equiv
		(\alpha - \beta)\begin{pmatrix}
			x_0 \\
			x_1
		\end{pmatrix} \pmod{\mathfrak{p}^v} \\
		&\quad\Longleftrightarrow\quad
		\begin{pmatrix}
			\alpha^n \\
			\beta^n
		\end{pmatrix}
		\equiv
		M^{-1}(\alpha - \beta)\begin{pmatrix}
			x_0 \\
			x_1
		\end{pmatrix} \pmod{\mathfrak{p}^v}
		\quad\Longleftrightarrow\quad
		\begin{pmatrix}
			\alpha^n \\
			\beta^n
		\end{pmatrix}
		\equiv
		\begin{pmatrix}
			1 \\
			1
		\end{pmatrix} \pmod{\mathfrak{p}^v} .
	\end{align*}
	Therefore, we get that $\tau(\bm{x}; p^v) = \lcm\!\big(\!\ord_{\mathfrak{p}^v}(\alpha), \ord_{\mathfrak{p}^v}(\beta)\big)$.
	Then, since $p$ is odd, the claim follows from Lemma~\ref{lem:orders}\ref{ite:orders:3}.
\end{proof}

\begin{lemma}\label{lem:case-n-odd-D-p-equal-1}
	Let $n \geq 3$ be integer.
	Suppose that $p$ is an odd primitive divisor of $\ell_n$ such that $(D \mid p) = 1$.
	Then there exists $\bm{x} \in \mathcal{L}(f)$ such that $\tau(\bm{x}; p) = \rho(\bm{x}; p) = 2n$ and all the residues of $\bm{x}$ modulo $p$ are nonzero.
	Moreover, if $n$ is odd, then there exists $\bm{y} \in \mathcal{L}(f)$ such that $\tau(\bm{y}; p) = \rho(\bm{y}; p) = n$ and all the residues of $\bm{y}$ modulo $p$ are nonzero.
\end{lemma}
\begin{proof}
	Let $\mathfrak{p}$ be a prime ideal of $\mathcal{O}_K$ lying over $p$.
	Since $p$ is a primitive divisor of $\ell_n$, by Lemma~\ref{lem:prim-div-order-alpha2} we have that $\ord_{\mathfrak{p}}(\alpha^2) = n$.
	Since $p$ is odd and $(D \mid p) = 1$, we have that $f$ splits modulo $p$, that is, there exist integers $a,b$ such that $a \equiv \alpha \pmod{\mathfrak{p}}$ and $b \equiv \beta \pmod{\mathfrak{p}}$.
	In particular, we have that $\ord_p(a) = \ord_{\mathfrak{p}}(\alpha)$ and $\ord_p(b) = \ord_{\mathfrak{p}}(\beta)$, where $\ord_p$ denotes the multiplicative order modulo $p$.
	By Lemma~\ref{lem:orders}\ref{ite:orders:1}--\ref{ite:orders:2}, we get that $\ord_p(a) = 2n$ or $\ord_p(b) = 2n$.
	By swapping $a$ and $b$, we can assume that $\ord_p(a) = 2n$.
	Let $\bm{x} \in \mathcal{L}(f)$ such that $x_0 = 1$ and $x_1 = a$.
	Since $f(a) \equiv 0 \pmod p$, it follows easily by induction that $x_n \equiv a^n \pmod p$ for every integer $n \geq 0$.
	Hence, we get that $\tau(\bm{x}; p) = \rho(\bm{x}; p) = 2n$ and all the residues of $\bm{x}$ modulo $p$ are nonzero, as desired.

	If $n$ is odd, then by Lemma~\ref{lem:orders}\ref{ite:orders:1}, we have that $\ord_p(b) = n$.
	Let $\bm{y} \in \mathcal{L}(f)$ such that $y_0 = 1$ and $y_1 = b$.
	Then, reasoning as for $\bm{x}$, we get that $\tau(\bm{y}; p) = \rho(\bm{y}; p) = n$ and all the residues of $\bm{y}$ modulo $p$ are nonzero, as desired.
\end{proof}

\begin{lemma}\label{lem:zero-term}
    Let $\bm{x} \in \mathcal{L}(f)$ with $x_0 = 1$, let $p$ be a prime number not dividing $D$, let $\mathfrak{p}$ be a prime ideal of $\mathcal{O}_K$ lying over $p$, let $v \geq 1$ be an integer, let $n := \ord_{\mathfrak{p}}(\alpha^2)$, and let $s \geq 0$ be an integer.
    Then we have that $x_s \equiv 0 \pmod{p^v}$ if and only if
    \begin{equation}\label{equ:zero-term:1}
        x_1 \equiv \frac{\beta - \alpha y_s}{1 - y_s} \pmod {\mathfrak{p}^v}
    \end{equation}
    and $1 - y_s$ is invertible modulo $\mathfrak{p}^v$, where $y_s := (-\alpha^2)^{-s}$.
\end{lemma}
\begin{proof}
    Since $p \nmid D$, we get that $\alpha - \beta$ is invertible modulo $\mathfrak{p}^v$.
    Hence, employing Lemma~\ref{lem:binet} and the fact that $\beta = - \alpha^{-1}$, it follows easily that $x_s \equiv 0 \pmod{p^v}$ is equivalent to
    \begin{equation}\label{equ:zero-term:2}
        x_1 (1 - y_s) \equiv \beta - \alpha y_s \pmod {\mathfrak{p}^v} .
    \end{equation}
	We claim that if \eqref{equ:zero-term:2} holds then $1 - y_s$ is invertible modulo $\mathfrak{p}^v$.
    Indeed, if $1 - y_s$ is not invertible modulo $\mathfrak{p}^v$, then $y_s \equiv 1 \pmod{\mathfrak{p}}$, and so \eqref{equ:zero-term:2} implies that $\alpha - \beta \equiv 0 \pmod{\mathfrak{p}}$, which is a contradiction, since $p \nmid D$.
	Hence, we have that $1 - y_s$ is invertible modulo $\mathfrak{p}^v$.
    Consequently, we get that \eqref{equ:zero-term:2} is equivalent to \eqref{equ:zero-term:1} and $1 - y_s$ being invertible modulo $\mathfrak{p}^v$, as desired.
\end{proof}

\begin{lemma}\label{lem:coincidence}
	Let $\bm{x} \in \mathcal{L}(f)$ with $x_0 = 1$, let $p$ be a prime number not dividing $D$, let $\mathfrak{p}$ be a prime ideal of $\mathcal{O}_K$ lying over $p$, let $v \geq 1$ be an integer, let $n := \ord_{\mathfrak{p}}(\alpha^2)$, and let $s, t$ be integers such that $0 \leq s < t < 2n$ and $s \equiv t \pmod 2$.
	Then we have that $x_s \equiv x_t \pmod{p^v}$ if and only if
	\begin{equation}\label{equ:coincidence:5}
		x_1 \equiv \frac{\beta - \alpha z_{s,t}}{1 - z_{s,t}} \pmod {\mathfrak{p}^v}
	\end{equation}
	and $1 - z_{s,t}$ is invertible modulo $\mathfrak{p}^v$, where $z_{s,t} := (-1)^{t + 1} \alpha^{-(s + t)}$.
\end{lemma}
\begin{proof}
	Since $p \nmid D$, we get that $\alpha - \beta$ is invertible modulo $\mathfrak{p}^v$.
	Hence, employing Lemma~\ref{lem:binet}, we have that $x_s \equiv x_t \pmod{p^v}$ is equivalent to
	\begin{equation}\label{equ:coincidence:2}
		(x_1 - \beta)\alpha^s - (x_1 - \alpha) \beta^s \equiv (x_1 - \beta)\alpha^t - (x_1 - \alpha) \beta^t \pmod {\mathfrak{p}^v} .
	\end{equation}
	Since $0 \leq s < t < 2n$, $s \equiv t \pmod 2$, and $n := \ord_{\mathfrak{p}}(\alpha^2)$, we have that $1 - \alpha^{t-s}$ is invertible modulo $\mathfrak{p}^v$.
	Hence, it follows that \eqref{equ:coincidence:2} is equivalent to
	\begin{equation}\label{equ:coincidence:3}
		x_1 - \beta \equiv \frac{\beta^t(\beta^{s-t} - 1)}{\alpha^s (1 - \alpha^{t-s})} \, (x_1 - \alpha) \pmod {\mathfrak{p}^v} .
	\end{equation}
	Since $\beta = -\alpha^{-1}$ and $s \equiv t \pmod 2$, we get that
	\begin{equation*}
		\frac{\beta^t(\beta^{s-t} - 1)}{\alpha^s (1 - \alpha^{t-s})}
		= (-1)^{t + 1} \alpha^{-(s + t)} = z_{s,t} .
	\end{equation*}
	Therefore, we have that~\eqref{equ:coincidence:3} is equivalent to
	\begin{equation}\label{equ:coincidence:4}
		x_1 (1 - z_{s,t}) \equiv \beta - \alpha z_{s,t} \pmod {\mathfrak{p}^v} .
	\end{equation}
	We claim that if \eqref{equ:coincidence:4} holds then $1 - z_{s,t}$ is invertible modulo $\mathfrak{p}^v$.
    Indeed, if $1 - z_{s,t}$ is not invertible modulo $\mathfrak{p}^v$, then $z_{s,t} \equiv 1 \pmod{\mathfrak{p}}$, and so \eqref{equ:coincidence:4} implies that $\alpha - \beta \equiv 0 \pmod{\mathfrak{p}}$, which is a contradiction, since $p \nmid D$.
	Hence, we have that $1 - z_{s,t}$ is invertible modulo $\mathfrak{p}^v$.
	Consequently, we get that \eqref{equ:coincidence:4} is equivalent to \eqref{equ:coincidence:5} and $1 - z_{s,t}$ being invertible modulo $\mathfrak{p}^v$, as desired.
\end{proof}

\begin{lemma}\label{lem:parity}
	Let $n \geq 4$ be an even integer.
	Suppose that $p$ is a high primitive divisor of $\ell_n$ such that $(D \mid p) = -1$, and put $v := \nu_p(\ell_n)$.
	Then there exists $\bm{x} \in \mathcal{L}(f)$ such that $\tau(\bm{x}; p^v) = 2n$, $x_s \not\equiv x_t \pmod {p^v}$ for all integers $s, t$ with $0 < s < t < 2n$ and $s \equiv t \pmod 2$, and all the residues of $\bm{x}$ modulo $p^v$ are nonzero.
\end{lemma}
\begin{proof}
	Let $\mathfrak{p}$ be a prime ideal of $\mathcal{O}_K$ lying over $p$.
	Since $p$ is a primitive divisor of $\ell_n$ and $p^v \mid \ell_n$, from Lemma~\ref{lem:prim-div-order-alpha2} we have that $n = \ord_{\mathfrak{p}}(\alpha^2) = \ord_{\mathfrak{p}^v}(\alpha^2)$.

	We claim that there exists an integer $c$ that satisfies
	\begin{equation}\label{equ:excluded:1}
		c \not\equiv \frac{\beta - \alpha^{2k+1}}{1 - \alpha^{2k}} \pmod{\mathfrak{p}^v}
	\end{equation}
	for every integer $k \geq 0$ such that $1 - \alpha^{2k}$ is invertible modulo ${\mathfrak{p}^v}$.
	Indeed, we have $p^v$ choices for $c$ modulo $\mathfrak{p}^v$, while the right-hand side of \eqref{equ:excluded:1} takes at most $n = \ord_{\mathfrak{p}^v}(\alpha^2)$ values modulo $\mathfrak{p}^v$.
    Since $p$ is a high primitive divisor of $\ell_n$, we have that $p^v > n$ and so it is possible to choose a value of $c$ with the desired property.

	Let $\bm{x} \in \mathcal{L}(f)$ with $x_0 = 1$ and $x_1 = c$.
	Since $(D \mid p) = -1$, we have that $p \nmid D$ and that $f$ has no roots modulo $p$.
	Hence, also since $p$ is odd, we get that $p \nmid 2(x_1^2 - a_1 x_1 - 1)D$.
	Consequently, by Lemma~\ref{lem:period}, we get that $\tau(\bm{x}; p^v) = 2n$, as desired.

	Since $n = \ord_{\mathfrak{p}^v}(\alpha^2)$, we have that $\alpha^{2n} \equiv 1 \pmod{\mathfrak{p}^v}$.
    Consequently, by Lemma~\ref{lem:sqrt-1}, we get that $\alpha^n \equiv \pm 1 \pmod{\mathfrak{p}^v}$.
    In fact, since $n$ is even and $n = \ord_{\mathfrak{p}^v}(\alpha^2)$, it follows that $\alpha^n \equiv -1 \pmod{\mathfrak{p}^v}$.
    Hence, we have that $-1$ is equal to a power of $\alpha^2$ modulo $\mathfrak{p}^v$.

	Suppose that there exist integers $s, t$ such that $0 < s < t < 2n$, $s \equiv t \pmod 2$, and $x_s \equiv x_t \pmod {p^v}$.
	By Lemma~\ref{lem:coincidence}, we get that
	\begin{equation}\label{equ:excluded:2}
		c \equiv x_1 \equiv \frac{\beta - \alpha z_{s,t}}{1 - z_{s,t}} \pmod{\mathfrak{p}^v} ,
	\end{equation}
	where $z_{s,t} := (-1)^{t + 1} \alpha^{-(s + t)}$ is equal to a power of $\alpha^2$ modulo $\mathfrak{p}^v$.
	But then \eqref{equ:excluded:1} and \eqref{equ:excluded:2} are in contradiction.
	Therefore, it follows that $x_s \not\equiv x_t \pmod {p^v}$ for all integers $s, t$ with $0 < s < t < 2n$ and $s \equiv t \pmod 2$.

    Suppose that $x_s \equiv 0 \pmod{p^v}$ for some integer $s \geq 0$.
    Then, by Lemma~\ref{lem:zero-term}, we have that
	\begin{equation}\label{equ:excluded:3}
        c \equiv x_1 \equiv \frac{\beta - \alpha y_s}{1 - y_s} \pmod{\mathfrak{p}^v} ,
    \end{equation}
    where $y_s := (-\alpha^2)^{-s}$ is equal to a power of $\alpha^2$ modulo $\mathfrak{p}^v$.
	But then \eqref{equ:excluded:1} and \eqref{equ:excluded:3} are in contradiction.
    Therefore, it follows that $x_s \not\equiv 0 \pmod {p^v}$ for all integers $s \geq 0$.
\end{proof}

For every $\bm{x} \in \mathcal{L}(f)$ and for all integers $m, d \geq 1$ and $r$, let us define
\begin{equation*}
	\rho(\bm{x}; m, r, d) := \#\big\{x_n \bmod m : n\geq 0,\; n \equiv r \!\!\!\pmod d \big\} .
\end{equation*}
We need the following lemma.

\begin{lemma}\label{lem:combining-sequences}
	Let $m_1, m_2 \geq 1$ be coprime integers, let $\bm{x}^{(1)}, \bm{x}^{(2)} \in \mathcal{L}(f)$, and put $\tau_i := \tau(\bm{x}^{(i)}; m_i)$ for each $i \in \{1,2\}$.
	If $\rho(\bm{x}^{(2)}; m_2) = \tau_2$ then there exists $\bm{x} \in \mathcal{L}(f)$ such that $\bm{x} \equiv \bm{x}^{(i)} \pmod {m_i}$ for each $i \in \{1,2\}$, and
	\begin{equation*}
		\rho(\bm{x}; m_1 m_2) = \frac{\tau_2}{d} \sum_{r = 0}^{d-1} \rho(\bm{x}^{(1)}; m_1, r, d) ,
	\end{equation*}
	where $d := \gcd(\tau_1, \tau_2)$.
\end{lemma}
\begin{proof}
	Since $m_1, m_2$ are coprime, we can pick $\bm{x} \in \mathcal{L}(f)$ such that $x_j \equiv x_j^{(i)} \pmod {m_i}$ for each $i \in \{1, 2\}$ and $j \in \{0,1\}$.
	Then, it follows easily by induction that $x_j \equiv x_j^{(i)} \pmod {m_i}$ for each $i \in \{1, 2\}$ and for every integer $j \geq 0$.
	At this point, the rest of the claim is a consequence of \cite[Lemma~6]{MR4125906}.
\end{proof}

\section{Proof of Theorem~\ref{thm:main}}\label{sec:proof-main}

We collected in Table~\ref{tab:special-cases} some values of $\tau(\bm{x}; m)$ and $\rho(\bm{x}; m)$, for $\bm{x} \in \mathcal{L}(f)$ and $m \in \mathbb{Z}^+$, that are needed for the proof.
We begin with the following lemma.

\begin{lemma}\label{lem:case-n-4}
    There exist $\bm{x} \in \mathcal{L}(f)$ and $m \in \mathbb{Z}^+$ such that $\rho(\bm{x}; m) = 4$ and all the residues of $\bm{x}$ modulo $m$ are nonzero.
    In particular, we have that $4 \in \mathcal{R}(f)$.
\end{lemma}
\begin{proof}
    If $a_1 = 1$ or $a_1 = 2$, then the claim follows from rows 2 or 13 of Table~\ref{tab:special-cases}, respectively.

    If $a_1$ is odd and $a_1 \geq 3$, then pick $\bm{x} \in \mathcal{L}(f)$ with $x_0 = 1$, $x_1 = 2$, and let $m := 2 a_1$.
    Since $2 a_1 \equiv 0 \pmod m$ and $a_1^2 \equiv a_1 \pmod m$, we get that the first terms of $\bm{x}$ are congruent to
    \begin{equation*}
        1,\, 2,\, 1,\, a_1 + 2,\, a_1 + 1,\, a_1 + 2,\, 1,\, 2, \, \dots \pmod m .
    \end{equation*}
    Noting that $m \geq 6$ and $a_1 \not\equiv -2, -1, 0, 1 \pmod m$, it follows easily that $\tau(\bm{x}; m) = 6$,  $\rho(\bm{x}; m) = 4$ and that all the residues of $\bm{x}$ modulo $m$ are nonzero.

    If $a_1$ is even and $a_1 \geq 4$, then pick $\bm{x} \in \mathcal{L}(f)$ with $x_0 = 1$, $x_1 = 3$, and let $m := 2 a_1$.
    Since $a_1^2 \equiv 2 a_1 \equiv 0 \pmod m$, we get that the first terms of $\bm{x}$ are congruent to
    \begin{equation*}
        1,\, 3,\, a_1 + 1,\, a_1 + 3,\, 1,\, 3,\, \dots \pmod m .
    \end{equation*}
    Noting that $m \geq 8$ and $a_1 \not\equiv -3,-2,-1,0,2 \pmod m$, it follows easily that $\tau(\bm{x}; m) = \rho(\bm{x}; m) = 4$ and that all the residues of $\bm{x}$ modulo $m$ are nonzero.
\end{proof}

We have to prove that $\mathcal{R}(f) = \mathbb{Z}^+$.
First, we prove that $\{n, 2n\} \subseteq \mathcal{R}(f)$ for every odd integer $n \geq 1$.
We have that $1 \in \mathcal{R}(f)$, since $\rho(\bm{x}; 1) = 1$ for every $\bm{x} \in \mathcal{L}(f)$; and $2 \in \mathcal{R}(f)$, since $\rho(\bm{x}; 2) = 2$ for every $\bm{x} \in \mathcal{L}(f)$ with $x_0 \not\equiv x_1 \pmod 2$.
If $a_1 = 1$, or $a_1 = 3$, then $\{3, 6\} \subseteq \mathcal{R}(f)$ by rows 1 and 3, or 17 and 18, of Table~\ref{tab:special-cases}, respectively.
Hence, assume that $n \geq 3$ is an odd integer, and $a_1 \notin \{1, 3\}$ or $n \neq 3$.
By Lemma~\ref{lem:existence-of-primitive-divisor}, we have that $\ell_n$ has an odd primitive divisor $p$.
Moreover, since $n$ is odd, from Lemma~\ref{lem:legendre-D-p-equal-1} it follows that $(D \mid p) = 1$.
Therefore, from Lemma~\ref{lem:case-n-odd-D-p-equal-1} we obtain that $\{n, 2n\} \subseteq \mathcal{R}(f)$.

It remains to prove that $2n \in \mathcal{R}(f)$ for every even integer $n \geq 2$.
By Lemma~\ref{lem:case-n-4}, we have that $4 \in \mathcal{R}(f)$.
Hence, we can assume that $n \geq 4$.

If $a_1 = 1$ and $n \in \{4, 6, 8, 10, 12, 14, 18, 24\}$, or $a_1 = 2$ and $n \in \{4, 6, 12\}$, or $a_1 = 3$ and $n = 6$, then $2n \in \mathcal{R}(f)$ by rows 4--11, 14--16, or 19 of Table~\ref{tab:special-cases}, respectively.

Hence, assume that $n \geq 4$ is an even integer, and that $a_1 = 1$ and $n \notin \{4, 6, 8, 10,$ $12, 14, 18, 24\}$, or $a_1 = 2$ and $n \notin \{4, 6, 12\}$, or $a_1 = 3$ and $n \neq 6$, or $a_1 \geq 4$.
Then, by Lemma~\ref{lem:existence-of-primitive-divisor-high-power}, we have that $\ell_n$ has a high primitive divisor $p$.

If $(D \mid p) = 1$ then Lemma~\ref{lem:case-n-odd-D-p-equal-1} yields that $2n \in \mathcal{R}(f)$, as desired.
Hence, suppose that $(D \mid p) = -1$.
Thanks to Lemma~\ref{lem:parity}, there exists $\bm{x}^{(1)} \in \mathcal{L}(f)$ such that $\tau(\bm{x}^{(1)}; m_1) = 2n$ and $\rho(\bm{x}^{(1)}; m_1, r, 2) = n$  for each $r \in \{0, 1\}$, where $m_1 := p^{\nu_p(\ell_n)}$.
Note that, since $n$ is even, we also get that $\rho(\bm{x}^{(1)}; m_1, r, 4) = n / 2$ for each $r \in \{0, 1, 2, 3\}$.
We define $\bm{x}^{(2)} \in \mathcal{L}(f)$ and $m_2 \in \mathbb{Z}^+$ as follows.
If $a_1 = 1$ then we pick $\bm{x}^{(2)}$ and $m_2$ from rows 2 of Table~\ref{tab:special-cases}.
If $a_1 \geq 2$ then we let $x_0^{(2)} := 0$, $x_1^{(2)} := 1$, and put $m_2 := a_1$.
It follows easily that $\tau(\bm{x}^{(2)}; m_2) = \rho(\bm{x}^{(2)}; m_2) \in \{2, 4\}$.
Since $p$ is a primitive divisor of $\ell_n$, we have that $p \nmid a_1 D$.
Hence, we get that $m_1, m_2$ are coprime integers.
(Note that $m_2 = D = 5$ if $a_1 = 1$.)
Put $\tau_i := \tau(\bm{x}^{(i)}; m_i)$ for each $i \in \{1, 2\}$, and let $d := \gcd(\tau_1, \tau_2)$.
Note that, since $n$ is even, we have that $d = \tau_2 \in \{2, 4\}$.
Therefore, applying Lemma~\ref{lem:combining-sequences} we get that there exists $\bm{x} \in \mathcal{L}(f)$ such that
\begin{equation*}
	\rho(\bm{x}; m_1 m_2) = \frac{\tau_2}{d} \sum_{r = 0}^{d-1} \rho(\bm{x}^{(1)}; m_1, r, d) = 1 \cdot \sum_{r = 0}^{d-1} \frac{2n}{d} = 2n ,
\end{equation*}
so that $2n \in \mathcal{R}(f)$, as desired.
The proof that $\mathcal{R}(f) = \mathbb{Z}^+$ is complete.

At this point, note that for each integer $n \geq 4$, and for each integer $n \geq 3$ if $a_1 \geq 2$, the sequence $\bm{x} \in \mathcal{L}(f)$ and the modulo $m$ that we constructed so that $\rho(\bm{x}; m) = n$ have the additional property that all the residues of $\bm{x}$ modulo $m$ are nonzero.
This follows from Table~\ref{tab:special-cases}, Lemma~\ref{lem:case-n-odd-D-p-equal-1}, Lemma~\ref{lem:case-n-4}, and Lemma~\ref{lem:parity}--Lemma~\ref{lem:combining-sequences} (note that $\bm{x}$ has no zero modulo $m_1 m_2$ since $\bm{x}^{(1)}$ has no zero modulo $m_1$).
Moreover, if $a_1 \geq 2$, then picking $\bm{x} \in \mathcal{L}(f)$ with $x_0 = 1$, $x_1 = 1$, and $m := a_1$, we get that $\rho(\bm{x}; m) = 1$ and all the residues of $\bm{x}$ modulo $m$ are nonzero.
If $a_1 = 2$ or $a_1 \geq 3$, then picking $\bm{x}$ and $m$ as in row 12 of Table~\ref{tab:special-cases}, or taking $x_0 = 1$, $x_1 = 2$, and $m = a_1$, we get that $\rho(\bm{x}; m) = 2$ and all the residues of $\bm{x}$ modulo $m$ are nonzero.

It remains to prove that if $a_1 = 1$ and $n \in \{1, 2, 3\}$ then there exist no $\bm{x} \in \mathcal{L}(f)$ and $m \in \mathbb{Z}^+$ such that $\rho(\bm{x}; m) = n$ and all the residues of $\bm{x}$ modulo $m$ are nonzero.
This is done in the last paragraph of \cite[p.~122]{MR4125906}.

The proof is complete.

\section{Proof of Corollary~\ref{cor:fractional-parts}}
\label{sec:proof-fractional-parts}

If $a_1 = 1$ then the result is in fact~\cite[Corollary~2]{MR4125906}.
Hence, assume that $a_1 \geq 2$.

We order $\alpha$ and $\beta$ so that $\alpha = \big(a_1 + \sqrt{D}\,\big) / 2$ and $\beta = \big(a_1 - \sqrt{D}\,\big) / 2$.
In particular, note that $\alpha > 1$ and $-1 < \beta < 0$.

Let $k \geq 1$ be an integer.
By Theorem~\ref{thm:main}, there exist $\bm{x} \in \mathcal{L}(f)$ and $m \in \mathbb{Z}^+$ such that $\rho(\bm{x}; m) = k$ and $x_n \not\equiv 0 \pmod m$ for every integer $n \geq 0$.
Moreover, by Lemma~\ref{lem:binet}, we have that $x_n = c_1 \alpha^n - c_2 \beta^n$ for all integers $n \geq 0$, where
\begin{equation*}
    c_1 := \frac{x_1 - \beta x_0}{\alpha - \beta} \quad \text{ and } \quad c_2 := \frac{x_1 - \alpha x_0}{\alpha - \beta} .
\end{equation*}
By eventually replacing $\bm{x}$ with $-\bm{x} := (-x_n)_{n \geq 0}$, we can assume that $c_1 > 0$.
Furthermore, note that $c_2 \neq 0$.
Let $\xi := c_1 / m$, so that $\xi \in \mathbb{Q}(\alpha)$.
Then, we have that
\begin{equation}\label{equ:xi-alpha-n}
    \xi \alpha^n - \frac{x_n}{m} = \frac{c_2 \beta^n}{m} \to 0 ,
\end{equation}
as $n \to +\infty$, since $|\beta| < 1$.
Let $r_1, \dots, r_k \in \mathbb{Z}$, with $0 < r_1 < \dots < r_k < m$, be the residues of $\bm{x}$ modulo $m$.
From~\eqref{equ:xi-alpha-n} it follows easily that the set of limit points of $\big(\!\fracpart(\xi \alpha^n)\big)_{n \geq 0}$ is equal to $\{r_1 / m, \dots, r_k / m\}$.
Hence, we get that $\big(\!\fracpart(\xi \alpha^n)\big)_{n \geq 0}$ has exactly $k$ limit points.

The proof is complete.

\begin{table}[h]
	\begin{center}
		\begin{tabular}{cccccccc}
			\toprule
			row n. & $a_1$ & $x_0$ & $x_1$ & $m$ & $\tau(\bm{x}; m)$ & $\rho(\bm{x}; m)$ \\
			\midrule
			$1$ & $1$ & $0$ & $1$ & $3$ & $8$ & $3$ &  \\
			$2$ & $1$ & $1$ & $3$ & $5$ & $4$ & $4$ & $*$ \\
			$3$ & $1$ & $1$ & $3$ & $8$ & $12$ & $6$ & $*$ \\
			$4$ & $1$ & $1$ & $3$ & $10$ & $12$ & $8$ & $*$ \\
			$5$ & $1$ & $1$ & $3$ & $13$ & $28$ & $12$ & $*$ \\
			$6$ & $1$ & $1$ & $3$ & $17$ & $36$ & $16$ & $*$ \\
			$7$ & $1$ & $1$ & $3$ & $28$ & $48$ & $20$ & $*$ \\
			$8$ & $1$ & $1$ & $3$ & $26$ & $84$ & $24$ & $*$ \\
			$9$ & $1$ & $1$ & $3$ & $56$ & $48$ & $28$ & $*$ \\
			$10$ & $1$ & $1$ & $3$ & $52$ & $84$ & $36$ & $*$ \\
			$11$ & $1$ & $1$ & $3$ & $78$ & $168$ & $48$ & $*$ \\
			\midrule
			$12$ & $2$ & $1$ & $1$ & $4$ & $4$ & $2$ & $*$ \\
			$13$ & $2$ & $1$ & $1$ & $5$ & $12$ & $4$ & $*$ \\
			$14$ & $2$ & $1$ & $1$ & $28$ & $12$ & $8$ & $*$ \\
			$15$ & $2$ & $1$ & $1$ & $13$ & $28$ & $12$ & $*$ \\
			$16$ & $2$ & $1$ & $1$ & $39$ & $56$ & $24$ & $*$ \\
			\midrule
			$17$ & $3$ & $1$ & $1$ & $9$ & $6$ & $3$ & $*$ \\
			$18$ & $3$ & $1$ & $1$ & $8$ & $12$ & $6$ & $*$ \\
			$19$ & $3$ & $1$ & $1$ & $17$ & $16$ & $12$ & $*$ \\
			\bottomrule
		\end{tabular}
        \caption{Values of $\tau(\bm{x}; m)$ and $\rho(\bm{x}; m)$ for $\bm{x} \in \mathcal{L}(X^2 - a_1 X - 1)$ with initial values $x_0, x_1$.
        The symbol $*$ means that all the residues of $\bm{x}$ modulo $m$ are nonzero.}\label{tab:special-cases}
    \end{center}
\end{table}

\FloatBarrier

\bibliographystyle{amsplain-no-bysame}
\bibliography{main}

\end{document}